\documentclass{article}
\usepackage{amsmath, amssymb, amsthm, epsf, graphicx}

\newtheorem{theorem}{Theorem}

\newtheorem{corollary}{Corollary}

\newtheorem{question}{Question}

\def\bb #1{ {\mathbb #1} }
\def\c #1{ {\mathcal #1} }

\begin{document}
\title{On the $p$-adic meromorphy of the \\ function field height zeta function}

\author{C. Douglas Haessig \\
University of Rochester, New York}

\date{\today}
\maketitle

\begin{abstract}
In this brief note, we will investigate the number of points of
bounded (twisted) height in a projective variety defined over a
function field, where the function field comes from a projective
variety of dimension greater than or equal to 2. A first step in 
this investigation is to understand the $p$-adic analytic properties of the height zeta
function. In particular, we will show that for a large class of
projective varieties this function is $p$-adic meromorphic.
\end{abstract}

\section{Introduction}

{\bf Number Fields.} Let $K$ be a number field of degree $m$ and
denote by $H_K$ the relative multiplicative height function on the
$K$-rational points in projective space $\bb P^n(K)$. Consider the
number of $K$-rational points of height bounded above by $d$:
\[
N_d(K) := \# \{ x \in \bb P^n(K) | H_K(x) \leq d \}.
\]
An asymptotic estimate for $N_d(K)$ was provided by Schanuel
\cite{Shanuel}, who proved
\[
N_d(K) = a(K,n) d^{n+1} +
\begin{cases}
\c O(d \log d) & \text{if } K = \bb Q \text{ and } n=1 \\
\c O(d^{n+1-\frac{1}{m}}) & \text{otherwise}
\end{cases}
\]
where
\[
a(K,n) = \frac{h R/\omega}{\zeta_K(n+1)} \left( \frac{2^{r_1}
(2\pi)^{r_2}}{\sqrt{ |D_K| }} \right)^{n+1} (n+1)^{r_1 + r_2 -1},
\]
$h$ is the class number of $K$, $R$ the regulator, $\omega$ the
number of roots of unity in $K$, $\zeta_K$ the Dedekind zeta
function of $K$, $D_K$ the discriminant, and $r_1, r_2$ the number
of real and complex embeddings of $K$.

\bigskip\noindent{\bf Function Fields.} A function field analogue of
Schanuel's results is also known. Let $\bb F_q$ be the finite field
with $q$ elements, $q = p^r$. Let $X$ be a non-singular projective
curve of genus $g$ defined over $\bb F_q$, and let $\bb F_q(X)$
denote its function field. With an $\bb F_q(X)$-rational point $y :=
[y_0 : y_1 : \cdots : y_n] \in \bb P^n(\bb F_q(X))$ we may define
the (logarithmic) height
\[
h_X(y) := - deg \inf_i (y_i)
\]
where $\inf_i (y_i)$ denotes the greatest divisor $D$ of $X$ such
that $D \leq (y_i)$ for all $i$ with $y_i \not= 0$. As before,
consider the number of $\bb F_q(X)$-rational points of height
bounded above by $d$:
\[
N_d(\bb P^n) := \# \{ y \in \bb P^n(\bb F_q(X)) | h_X(y) \leq d \}.
\]
An asymptotic estimate for $N_d(\bb P^n)$ was described by Serre
\cite[p.19]{Serre}, and improved upon by Wan \cite{WanHieght} as
follows: for any $\epsilon > 0$,
\begin{equation}\label{E: WanAsymp}
N_d(\bb P^n) = \frac{h q^{(n+1)(1-g)}}{\zeta_X(n+1) (q-1)}q^{(n+1)d}
+ \c O(q^{\frac{d}{2}+\epsilon}).
\end{equation}

Wan's proof began by demonstrating that the associated function
field height zeta function of $\bb P^n$ is a rational function. This
meant, for all $d$, $N_d(\bb P^n) = \sum \beta_j^d - \sum
\alpha_i^d$ where $\alpha_i$ and $\beta_j$ are the zeros and poles
of this rational function. The next step was to analyze these zeros
and poles to arrive at (\ref{E: WanAsymp}).

\bigskip\noindent{\bf Function Field Variations.} Three
natural variations of the above study are:
\begin{enumerate}
\item Consider $X$ with $dim(X) \geq 2$,
\item Replace $\bb P^n$ with an arbitrary projective variety $Y$
defined over $\bb F_q(X)$,
\item Twist the height function $h_X$ by an element $A \in
GL_{n+1}(\bb F_q(X))$.
\end{enumerate}

Let us be more precise. Let $X$ be a projective integral scheme $X$
of finite type over $\bb F_q$ and denote by $\bb F_q(X)$ its
function field. Fix $A \in GL_{n+1}(\bb F_q(X))$. With an $\bb
F_q(X)$-rational point $y := [y_0 : y_1 : \cdots : y_n] \in \bb
P^n(\bb F_q(X))$ we may define the twisted (logarithmic) height
\[
h_{X, A}(y) := - deg \inf (Ay)
\]
where $\inf (Ay)$ means we are taking the infimum of the divisors of
the coordinates of $Ay$. Let $Y \hookrightarrow \bb P^n_{\bb
F_q(X)}$ be a projective variety defined over $\bb F_q(X)$. Define
the (twisted) height zeta function of $Y$, with respect to $A$, as
\[
Z_\text{ht}(Y(\bb F_q(X)), A, T) := \sum_{d=0}^\infty N_d(A,Y) T^d
\]
where
\[
N_d(A,Y) := \# \{ y \in Y(\bb F_q(X)) | h_{X, A}(y) \leq d \}.
\]
When $A$ is the identity matrix, we will simply write
$Z_\text{ht}(Y(\bb F_q(X)), T)$ for the height zeta function.

The first step in the study of $N_d(A, Y)$ is to determine the
($p$-adic) analytic behavior of the height zeta function. When $Y =
\bb P^n_{\bb F_q(X)}$ and the monoid of effective divisor classes of
$X$ has rank 1, Wan \cite{WanHieght} proved that $Z_\text{ht}(\bb
P^n(\bb F_q(X)), T)$ is $p$-adic meromorphic, and rational when
$dim(X) = 1$.

In this paper, under the same assumption that the monoid of
effective divisor classes of $X$ has rank 1, we will demonstrate
that for a large class of projective varieties $Y$ defined over $\bb
F_q(X)$ with $dim(X) \geq 2$, the twisted height zeta function is
$p$-adic meromorphic. As an example, suppose $X$ is a complete
intersection with $dim(X) \geq 3$ and consider the projective
hypersurface $Y$ defined by $F := y_0^r G(y_1, \ldots, y_n)$ where
$G$ is a homogenous polynomial in $\bb F_q(X)[y_1, \ldots, y_n]$ of
degree $d$. As a consequence of Theorem \ref{T: Reduction} and
Theorem \ref{T: Main}, if $n > d^{dim(X) +1}$, then
$Z_\text{ht}(Y(\bb F_q(X)), T)$ is $p$-adic meromorphic.

Our approach will be to reduce the study of the height zeta function
to the study of two other zeta functions: the zeta functions of
divisors and the Riemann-Roch zeta function. The zeta function of
divisors has been studied in \cite{WanZetaCycles}
\cite{WanPureLfunctions} \cite{HaessigWanDivisors} and is expected
to have a good $p$-adic theory. The Riemann-Roch zeta function is
new, and section \ref{S: RRZeta} explores a few of its properties.

\section{Reduction of the height zeta function}

In this section, we will reduce the study of the function field
height zeta function to the study of two other zeta functions,
namely the zeta function of divisors and the Riemann-Roch zeta
function.

Define $Div(X / \bb F_q)$ as the free abelian group generated by the
irreducible $\bb F_q$-subvarieties of $X$ of codimension 1; we call
these prime divisors. We say a divisor $D \in Div(X / \bb F_q)$ is
effective if its unique decomposition into prime divisors $P_i$,
written $D = \sum n_i P_i$, has $n_i \geq 0$ for each $i$. Let
$E(X)$ denote the monoid of effective divisors of $X$. Effectiveness
places a partial order on the group $Div(X / \bb F_q)$ as follows:
we say $E \geq D$ if $E - D$ is effective. Further, for any set of
divisors $\{D_1, \ldots D_k \}$, we may define $\sup_i D_i$ as the
smallest divisor $D$ such that every $D - D_i$ is effective.
Similarly, we may define $\inf_i D_i$ as the greatest divisor $D$
such that every $D_i - D$ is effective.

The \emph{zeta function of divisors} of $X$ is defined by
\[
Z(X, T) := \sum_{D \in E(X)} T^{deg(D)} = \prod_{P \text{ prime
divisor}} \frac{1}{1 - T^{deg(P)}}
\]
where $deg(D)$ is the projective degree of the divisor; see
\cite{HaessigWanDivisors} for the definition. Note, when $X$ is a
curve, then $Z(X,T)$ is the usual Weil zeta function, and hence is
rational. When $dim(X) \geq 2$, then $Z(X, T)$ is never rational
\cite[Theorem 3.3]{WanZetaCycles}.

Let $A(X)$ denote the group of linear equivalence classes of
divisors of $X$. Define $A^+(X) \subset A(X)$ as the monoid of
effective divisor classes. Wan \cite{WanZetaCycles} has conjectured
that when the rank of $A^+(X)$ is finite, then $Z(X,T)$ is $p$-adic
meromorphic. This was proven by Wan \cite[Theorem
4.3]{WanZetaCycles} when the rank is one. Examples of $p$-adic
meromorphy are also known for rank greater than one.

Next, let $W \hookrightarrow \bb A^n_{\bb F_q(X)}$ be an affine variety 
defined over $\bb F_q(X)$. For each divisor
$D$ of $X$, define the Riemann-Roch space $L(D)$ as the finite
dimensional $\bb F_q$-vector space
\[
L(D) := \{ f \in \bb F_q(X) | (f) + D \geq 0 \} \cup \{ 0 \}.
\]
Denote its dimension by $l(D)$. The \emph{Riemann-Roch zeta function} 
of $W$ over $\bb F_q(X)$ is defined by
\[
Z_{RR}(W(\bb F_q(X)), T) := \sum_{D \in E(X)} S_D T^{deg(D)}
\]
where $S_D$ denotes the number of $\bb F_q(X)$-rational points of
$W$ whose affine coordinates lie in $L(D)$; that is,
\[
S_D := \# \left( L(D)^n \cap W \right) = \# \{ f := (f_1, \ldots,
f_n) \in L(D)^n | f \in W \}.
\]
Our main theorem in this paper, proven in the next section, will
show that the Riemann-Roch zeta function is often $p$-adic
meromorphic.

\begin{question}
Assuming the rank of $A^+(X)$ is finite, is the Riemann-Roch zeta function always $p$-adic meromorphic?
\end{question}

\begin{question}
If $X$ is a curve, is the Riemann-Roch zeta function always rational?
\end{question}

We are now able to reduce the height zeta function as a sum of the
two zeta functions above. First, recall the following decomposition
of projective space. Denote by $y = [y_0 : y_1 : \cdots : y_n]$ the
points of $\bb P^n_{\bb F_q(X)}$. Let $\c S_{n+1}$ denote the
symmetric group on the set $\{ 0, 1, \ldots, n\}$ and fix $\sigma
\in \c S_{n+1}$. Let $A \in GL_{n+1}(\bb F_q(X))$ and denote its
rows by $A_i$. For each $i$, define the affine space $\bb
A^{n-i}_{\bb F_q(X)} \hookrightarrow \bb P^n_{\bb F_q(X)}$ as the
variety defined by setting $A_{\sigma(0)} y = \cdots
A_{\sigma(i-1)}y = 0$ and $A_{\sigma(i)}y = 1$. We then have the
decomposition of projective space into a disjoint union:
\[
\bb P^n_{\bb F_q(X)} = \bb A^n_{\bb F_q(X)} \cup \bb A^{n-1}_{\bb
F_q(X)} \cup \cdots \cup \bb A^0_{\bb F_q(X)}.
\]
Let $Y \hookrightarrow \bb P^n_{\bb F_q(X)}$ be a projective
variety. For each $i$ we may define the affine algebraic variety
$Y_i := Y \cap \bb A^i_{\bb F_q(X)}$. To keep the notation clean, we
have neglected to indicate notationally the dependence of $Y_i$ on
$A$ and the choice of $\sigma$.

\begin{theorem}\label{T: Reduction}
We may decompose the height zeta function as follows:
\[
Z_\text{ht}(Y(\bb F_q(X)), A, T) = \frac{1}{Z(X,T)} \sum_{i = 0}^n
Z_{RR}(Y_i(\bb F_q(X)), T).
\]
\end{theorem}

\noindent {\bf Remark.} The decomposition in the theorem depends on $A$ and the
choice of $\sigma \in \c S_{n+1}$. This flexibility often allows one
to prove $Z_\text{ht}(Y(\bb F_q(X)), A, T)$ is $p$-adic meromorphic
for specific examples.

\begin{proof}
Since every coordinate of $y = [y_0 : y_1 : \cdots : y_n] \in \bb P_{\bb
F_q(X)}^n$ cannot be simultaneously zero, we may assume $y_i=1$ for
some $i$. It follows that $- \inf_i (y_i) = \sup_i (y_i)_\infty$,
where $(y_i)_\infty$ is the polar divisor of $y_i \in \bb F_q(X)$.
Thus,
\begin{align*}
Z_\text{ht}(Y(\bb F_q(X)), A, T) &= \sum_{[y_0 : \cdots : y_n] \in
Y} T^{-deg \inf (Ay)} \\
&= \sum_{i=0}^n \sum_{(y_0, \cdots, y_n) \in
Y_i} T^{deg \sup (Ay)_{\infty}} \\
&= \sum_{i=0}^n \sum_{D \in E(X)} H_D^{(i)} T^{deg(D)}
\end{align*}
where $H_D^{(i)} := \# \{ (y_0, \cdots, y_n) \in Y_i |
\sup(Ay)_\infty = D \}$; note, we put round-brackets to indicate
that the point $(y_0, \ldots, y_n)$ lies in the affine space $\bb A_{\bb F_q(X)}^i$. Now,
\[
Z(X, T) \left( \sum_{D \in E(X)} H_D^{(i)} T^{deg(D)} \right) =
\sum_{ E \in E(X) } \left( \sum_{0 \leq D \leq E} H_D^{(i)} \right)
T^{deg(E)}.
\]
The theorem follows since $S_E = \sum_{0 \leq D \leq E} H_D^{(i)}$.
\end{proof}

\section{The Riemann-Roch zeta function}\label{S: RRZeta}

In this section, we will study the $p$-adic meromorphy of the
Riemann-Roch zeta function. Our main result will be to show, for a
large class of affine varieties, that it is $p$-adic entire. After
this, we present some explicit examples.

\begin{theorem}\label{T: Main}
Suppose $A^+(X)$ is of rank one and $dim(X) \geq 2$. Let $Y
\hookrightarrow \bb A^n_{\bb F_q(X)}$ be an affine hypersurface
defined by a polynomial $F \in \bb F_q(X)[y_1, \ldots, y_n]$. Define
$d := deg(F)$. If $n > d^{dim(X) + 1}$, then $Z_{RR}(Y(\bb F_q(X)),
T)$ is a $p$-adic entire function.
\end{theorem}

\begin{proof}
We begin by writing
\[
Z_{RR}(Y(\bb F_q(X)), T) := \sum_{D \in E(X)} S_D T^{deg(D)} =
\sum_{k \geq 0} M_k T^k
\]
where
\begin{equation}\label{E: Mk}
M_k := \sum_{D \in E(X), deg(D) = k} S_D.
\end{equation}
We will show that under the given hypothesis, for all $k$
sufficiently large, $ord_q(S_D)$ is bounded below by a polynomial of
degree $dim(X)$. Hence, the same is true for $ord_q(M_k)$ proving
the theorem.

Decomposing $A(X)$ into its torsion part and free part, we have
$A(X) = \{D_1^*, \ldots, D_h^* \}_{tor} \oplus \bb Z D$ for some
divisor $D > 0$. Let $\mu := deg(D)$. Thus, every effective divisor
$E$ of $X$ has degree $k \mu$ for some $k \in \bb Z_{\geq 0}$.
Further, if $deg(E) = k \mu$, then $E$ is linearly equivalent to
$D_j^* + kD$ for some $j$. By \cite[Proposition
2]{HaessigWanDivisors}, for $k \gg 0$ and $deg(E) = k \mu$, the dimension of the Riemann-Roch space satisfies:
\[
l(E) = l(D_j^* + k D) = c k^{dim(X)} + O(k^{dim(X)-1})
\]
where $c := \frac{D^{dim(X)}}{dim(X)!}$. From now on, we will
assume $k$ is large enough so that the above formula holds for any
effective divisor of degree $k \mu$.

Next, let us write $F = \sum_{j=1}^J \alpha_j y^{v_j}$ where
$\alpha_j \in \bb F_q(X)$ and $v_j \in \bb Z_{\geq 0}^n$. With the effective divisor $W :=
\sup_j (\alpha_j)_\infty$, we have $\alpha_j \in L(W)$ for every
$j=1,\ldots, J$. Find $i$ and $\rho$ such that $D_i^* + \rho D$ is
linearly equivalent to $W$. Using this, we observe that if $(f_1,
\ldots, f_n) \in L(E)^n$, then $F(f_1, \ldots, f_n) \in L(dE + W)$.

Let $u_1, \ldots, u_s$ be a basis of $L(E)$ over $\bb F_q$. Thus,
for every $h \in L(E)$, there exists unique $x_i \in \bb F_q$ such
that $h = x_1 u_1 + \cdots + x_s u_s$, and vice versa. Using this structure, we may
rewrite the polynomial $F(y_1,\ldots, y_n)$ with the substitution
\begin{align*}
&y_1 = x_1^{(1)} u_1 + \cdots + x_s^{(1)} u_s \\
&y_2 = x_1^{(2)} u_1 + \cdots + x_s^{(2)} u_s \\
&\quad \vdots \qquad \vdots \\
&y_n = x_1^{(n)} u_1 + \cdots + x_s^{(n)} u_s.
\end{align*}
Thus, we now have $F$ as a function in the variables $x_i^{(j)}$ for
$1 \leq i \leq s$ and $1 \leq j \leq n$. If we let $w_1, \ldots, w_r$ be a basis of $L(dE + W)$, then we may write
\[
F(x_i^{(j)}) = g_1(x_i^{(j)}) w_1 + \cdots + g_r(x_i^{(j)}) w_r
\]
for some $g_i \in \bb F_q[x_i^{(j)} : 1 \leq i \leq s$ and $1 \leq
j \leq n]$. Notice that $deg(g_i) \leq d$. Writing $F$ in this
form means that elements of $S_E$ are precisely the elements in the
zero locus of the set $\{g_1, \ldots, g_r\}$ in the affine space
$\bb A^{sn}_{\bb F_q}$. Therefore, by a theorem of Ax-Katz
\cite{WanElemKatz}
\[
ord_q(S_E) \geq \frac{ sn - \sum_{i=1}^r deg(g_i)}{\max deg(g_i)}.
\]
Since, $s = l(E)$ and $r = l(dE+W)$, and recalling that $E$ and $W$
are linearly equivalent to $D_j^* + kD$ and $D_i^* + \rho D$
respectively, we see that
\begin{align*}
ord_q(S_E) &\geq \frac{ sn - \sum_{i=1}^r deg(g_i)}{\max deg(g_i)}
\\
&\geq \frac{1}{d} \left[l(E) n - l(dE+W) d \right] \\
&= \frac{1}{d} \left[l(D_j^* + kD) n - l((dD_j^* + D_i^*) +
(dk+\rho)D) d \right]
\\
&= \frac{1}{d} \left[c k^{dim(X)} n - c(kd+\rho)^{dim(X)}
d \right] + O(k^{dim(X)-1}) \\
&= \frac{c k^{dim(X)}}{d}[n - d^{dim(X)+1}] + O(k^{dim(X)-1}).
\end{align*}
The theorem follows.
\end{proof}

\begin{corollary}\label{C: Main}
Suppose $A^+(X)$ is of rank one and $dim(X) \geq 2$. Let $Y \subset
\bb A^n_{\bb F_q(X)}$ be the affine variety defined by the
polynomials $F_1, \ldots, F_r \in \bb F_q(X)[y_1, \ldots, y_n]$.
Define $d_i := deg F_i$. Suppose $n
> \sum_{i=1}^r d_i^{dim(X)+1}$. Then $Z_{RR}(Y(\bb F_q(X)), T)$ is a
$p$-adic entire function.
\end{corollary}

\bigskip\noindent {\bf Example 1.} Consider affine $n$-space
$\bb A^n_{\bb F_q(X)}$. In this case, we have
\begin{equation}\label{E: AffineRR}
Z_{RR}(\bb A^n(\bb F_q(X)), T) = \sum_{D \geq 0} q^{n l(D)}
T^{deg(D)}.
\end{equation}
It follows that if $A^+(X)$ is of rank one, then this is a $p$-adic
entire function when $dim(X) \geq 2$ and a rational function when
$X$ is a curve.

Wan \cite{WanHieght} has conjectured that when the rank of $A^+(X)$
is finite, then the height zeta function of projective space $\bb
P^n_{\bb F_q(X)}$ is $p$-adic meromorphic. An equivalent conjecture
is that the Riemann-Roch zeta function of affine space $\bb A^n_{\bb
F_q(X)}$ is $p$-adic meromorphic.

\bigskip\noindent {\bf Example 2.}
Combining Theorem \ref{T: Main} or Corollary \ref{C: Main} with
Theorem \ref{T: Reduction}, we see that the height zeta function is
$p$-adic meromorphic for many projective varieties $Y$. Fix $A \in
GL_{n+1}(\bb F_q(X))$ and let $A_0$ denote its first row. As an
arbitrary example, consider the projective hypersurface $Y$
defined by $F := (A_0 y)^r G(y_1, \ldots, y_n)$ where $G$ is a
homogenous polynomial in $\bb F_q(X)[y_1, \ldots, y_n]$ of degree
$d$. If $A^+(X)$ is of rank one and $n > d^{dim(X) +1}$, then the
twisted height zeta function $Z_\text{ht}(Y(\bb F_q(X)), A, T)$ is
$p$-adic meromorphic.


\end{document}